\documentclass[11pt,a4paper]{amsart} 

\author{Thanatkrit Kaewtem}
\address{School of Mathematics\\
  University of Bristol\\
  Bristol BS8 1TW, UK} 
\email 
{tk13633@bristol.ac.uk}

\usepackage{amsmath,amssymb,amsthm,latexsym,graphicx,indentfirst,subfig, pb-diagram,wasysym} 
\usepackage{amsmath,amssymb,amsthm 
}
\usepackage{calrsfs}
\usepackage{fullpage}

\numberwithin{equation}{section}

\theoremstyle{plain}
\newtheorem{thm}{Theorem} 
\newtheorem{lem}[thm]{Lemma}
\newtheorem{cor}[thm]{Corollary}
\newtheorem{prop}[thm]{Proposition}

\theoremstyle{definition}
\newtheorem{defn}[thm]{Definition}
\newtheorem{example}[thm]{Example}
\newtheorem{rem}[thm]{Remark}

\newcommand{\R}{\mathbb{R}}
\newcommand{\N}{\mathbb{N}}

\usepackage[]{graphicx}
\title{Entropy numbers in $\gamma$-Banach spaces}

\begin{document}
\maketitle    

\begin{abstract}
Let $X$ be a quasi-Banach space, $Y$ a $\gamma$-Banach space $(0<\gamma \leq 1)$ and $T$ a bounded linear operator from $X$ into $Y$. 
In this paper, we prove that the first outer entropy number of $T$ lies between $2^{1-1/\gamma}\|T\|$ and $\|T\|$
; more precisely, $2^{1-1/\gamma}\|T\| \leq e_1(T) \leq \|T\|,$ and the constant $2^{1-1/\gamma}$ is sharp.
 Moreover, we show that there exist a Banach space $X_0$, a $\gamma$-Banach space $Y_0$ and a bounded linear operator $T_0:X_0 \rightarrow Y_0$ such that $0 \neq e_k(T_0) = 2^{1-1/\gamma}\|T_0\| $ for all positive integers $k.$ 
 Finally, the paper also provides two-sided estimates for entropy numbers of embeddings between finite dimensional symmetric $\gamma$-Banach spaces.
\end{abstract}

\section{Introduction}

Entropy numbers of a bounded linear operator $T$ provide a tool to measure the degree of compactness of $T.$
The ideas underlying the concept of entropy numbers go back a long way to the work of Pontryagin and Schnirelmann (1932), 
and Kolmogorov (1956), on the metric entropy of compact subsets of a metric space (see \cite{P2}).
An early remarkable result was that of Kolmogorov and Tikhomirov in 1959: this concerned the embedding of $C^k([0,1]^n)$ in $C([0,1]^n),$ where $k \in \N$ (see \cite{Kol&Ti}).  
Vitushkin and Henkin (1967) applied this idea in their work on the superposition of functions related to Hilbert's thirteenth problem (see \cite{Vitushkin&Henkin}).
The knowledge of entropy numbers of sets, and the corresponding idea of entropy numbers of bounded linear operators acting between Banach spaces, have been developed to such an extent that they are one of the important tools in analysis, especially in approximation theory.
For the background and applications of entropy numbers, we refer to any textbooks given in the reference, especially  \cite{P0,P1,C&S, E&T}.

For $\gamma \in (0,1],$ a {\it $\gamma$-norm} on a vector space $X$ is a function $\|\cdot | X\|:X \rightarrow [0,\infty)$ which satisfies the properties of norms, but instead of the triangle inequality, the inequality 
\begin{equation} \label{Eq-DefofGammaNorm}
	\|x+y | X\|^\gamma \leq \|x | X\|^\gamma + \|y | X\|^\gamma
\end{equation}
 holds for any $x,y \in X$.
 It is clear that every $\beta$-norm is a $\gamma$-norm for $0<\gamma \leq \beta \leq 1.$
 If (\ref{Eq-DefofGammaNorm}) is replaced by 
 \begin{equation}  \label{Eq-DefofQuasiNorm}
 \|x+y | X\| \leq C(\|x | X\|+ \|y | X\|)
 \end{equation}
  for some constant $C \geq 1$, then it is called a {\it quasi-norm}. 
According to the Aoki-Rolewicz theorem (\cite{Aoki} and \cite{Rolewicz}), if $X$ is a quasi-Banach space with a constant $C$, then there exists $\gamma \in (0,1]$ (in fact, $ 2^{1/\gamma-1} = C$) and a $\gamma$-norm on $X$ which is equivalent to the original quasi-norm. 

	Let $X$ be a quasi-Banach space, $Y$ a $\gamma$-Banach space and $T$ a bounded linear operator from $X$ into $Y$. 
	If $Y$ is a Banach space, it is well-known that  the first outer entropy number of $T$,  the first inner entropy number of $T$ and the norm of $T$ are identical; i.e.,  $e_1(T) = f_1(T) = \|T\|.$
In general, when $Y$ is a $\gamma$-Banach space, these numbers are related by 
\begin{equation} \label{Intro-First-Outter-Inner-Entropy-Estimate}
	2^{1-1/\gamma}\|T\| \leq e_1(T) \leq \|T\| \leq f_1(T) \leq  2^{1/\gamma-1}\|T\|
\end{equation}
(see Theorem \ref{SharpConstant2} and Theorem \ref{SharpConstant1}).
The constant $2^{1/\gamma-1}$ in (\ref{Intro-First-Outter-Inner-Entropy-Estimate}) is best possible.
Indeed, by considering the identity map $I:\ell_{\gamma} \rightarrow \ell_{\gamma}$, it can be demonstrated that $f_k(I) = 2^{1/\gamma -1} \|I\|$ for all $k \in \N$  (see Theorem \ref{SharpConstant2} below).
 The first aim of this paper is to show that there exist a Banach space $X$, a $\gamma$-Banach space $Y$ and a bounded linear operator $T:X \rightarrow Y$ such that $e_k(T) = 2^{1-1/\gamma}\|T\|$ for every $k \in \N$ (Theorem \ref{SharpConstant3}).
 This result implies that the constant $2^{1-1/\gamma}$ in (\ref{Intro-First-Outter-Inner-Entropy-Estimate}) is sharp.
Next, we consider the metric injection property of entropy numbers (see Definition \ref{Def-Metric-Injection}). 
Adapting the results in [2], p.125, we show that, for any $k \in \N,$ 
\begin{equation}  \label{Intro-Metric-Inj-Etimate}
	e_{k}(\iota T) \leq e_{k}(T) \leq 2^{1/\gamma}e_k(\iota T),
\end{equation} 
where $\iota$ is a metric injection, and the constant $2^{1/\gamma}$ in (\ref{Intro-Metric-Inj-Etimate}) cannot be reduced  (Proposition \ref{InjectivityOff} and Example \ref{Ex-constant-cannot-be-reduced}).
 
Let $X$ be a Banach space with a symmetric basis $(x_i)_{i=1}^\infty$ (see section 4 for the definition of symmetric base).
A {\it fundamental function} of $X,$ $\varphi_X,$ and a fundamental function of $X^\ast$ (the dual space of $X$),$ \varphi_{X^\ast},$ are defined, respectively, by 
\begin{equation}
	\varphi_X(m) := \left\| \sum_{i=1}^{m}x_i | X  \right\| ~\text{and}~ \varphi_X^\ast(m) := \left\| \sum_{i=1}^{m}x_i^{(\ast)} | X^\ast  \right\|
\end{equation}
for any $m \in \N.$
Here the functional $x_j^{(\ast)}$ is given by $x_j^{(\ast)} (x_i) = \delta_{i,j}.$  
It is well-known that 
\begin{equation} \label{Eq-Intro-ProductofFundaFn}
	\varphi_X(m) \varphi_{X^\ast}(m) = m
\end{equation}
 for all $m \in \N.$
In 1984, Sch\"{u}tt (see \cite{Sch}, Lemma 4) used the equation (\ref{Eq-Intro-ProductofFundaFn}) and volume arguments to prove that if $id:X \rightarrow Y$ is the natural embedding between any $n$-dimensional symmetric Banach spaces $X$ and $Y,$ then, for $k \geq n,$ there are absolute constants $c_1,c_2$ such that
	\begin{equation} \label{Eq-Intro-EntropyOfEmbeddingInFiniteDimB}
		c_1 2^{-k/n} \frac{\varphi_Y(n)}{\varphi_X(n)} \leq e_k(id:X \rightarrow Y) \leq c_2 2^{-k/n} \frac{\varphi_Y(n)}{\varphi_X(n)}.
	\end{equation}
 In the last section, we prove that (\ref{Eq-Intro-EntropyOfEmbeddingInFiniteDimB}) is still valid if $X$ and $Y$ are $n$-dimensional symmetric $\gamma$-Banach spaces (Theorem \ref{Thm-EntropyInFinteDim-FundamentalFn1}). 
 As we know that dual spaces of $\gamma$-Banach spaces might not have rich structures,
  to prove Theorem \ref{Thm-EntropyInFinteDim-FundamentalFn1},  different techniques are needed.
  For the sake of completeness, let us mention the case $k$ is small $(k \leq n).$
  If $X$ and $Y$ are $n$-dimensional symmetric Banach spaces, Sch\"{u}tt (see \cite{Sch}, Theorem 5) also provided two-sided estimates for $e_k(id:X \rightarrow Y);$ however, they were not sharp.
  Later, in 1998, the sharp two-sided estimates for $e_k(id:X \rightarrow Y),$ where $X$ and $Y$ are $n$-dimensional symmetric $\gamma$-Banach spaces,  were given by Edmunds and Netrusov   (see \cite{E&N1998}, Section 4, Theorem 2), but the case $k\geq n$ was left open. 
  The result in the paper fills this gap (see Theorem \ref{Thm-EntropyInFinteDim-FundamentalFn1}).


\section{Notations and preliminaries}

In this section we collect some basic facts, conventions and definitions that will be used later. 
All spaces considered will be assumed to be real vector spaces.
Given quasi-Banach spaces $X$ and $Y$, we write $\mathcal{B}(X,Y)$ for the space of all bounded linear maps from $X$ into $Y,$ and write $\mathcal{B}(X)$ if $X=Y.$ The symbol 
$B_X$ stands for the closed unit ball in $X$. 

Let $T \in \mathcal{B}(X,Y)$ and $k \in \N.$
The $k^{\text{th}}$ (dyadic) {\it outer entropy number} $e_k(T)$ of $T$ is defined to be the infimum of all those $\varepsilon >0$ such that $T(B_X)$ can be covered by $2^{k-1}$ balls in $Y$ with radius $\varepsilon.$
The numbers $e_k (T)$ are monotonic decreasing as $k$ increases, with $e_1 (T) \leq \|T\|.$ 
Moreover, $T$ is compact if and only if $\displaystyle \lim_{k \rightarrow \infty} e_k(T) = 0.$
If $Y$ is a $\gamma$-Banach space, then, for each $k_1,k_2 \in \N,$
$$e_{k_1+k_2-1}^\gamma (T_1+T_2) \leq e_{k_1}^\gamma (T_1) + e_{k_2}^\gamma (T_2) ~\text{and}~ e_{k_1+k_2-1} (RS) \leq e_{k_1}(R) e_{k_2}(S),$$
whenever $T_1+T_2$ and $RS$ are properly defined operators (see \cite{E&T}, p.7).
The  $k^{\text{th}}$ (dyadic) {\it inner entropy number} $f_k(T)$ of $T \in \mathcal{B}(X,Y)$ is defined to be the supremum of all those $\varepsilon > 0$ such that there are $x_1,...,x_{2^{k-1}+1} \in B_X$ with $\|Tx_i - Tx_j | Y\| \geq 2 \varepsilon$ whenever $i,j$ are distinct elements of $\{1,...,2^{k-1}+1\}.$
If $Y$ is a $\gamma$-Banach space, then the outer and inner entropy numbers are related by
 	  		\begin{equation} \label{EqRelationInner&Outer}
 	  			2 ^{1-1/\gamma}f_k(T) \leq e_k(T) \leq 2f_k(T).
 	  		\end{equation}
These estimates were proved by Pietsch (see \cite{P0}, p.169) in the Banach space case ($\gamma = 1$); a simple modification gives us (\ref{EqRelationInner&Outer}). 
 Throughout the paper the phrase \lq \lq entropy numbers\rq \rq ~always means outer entropy numbers.

   
   If $X$ is an $n$-dimensional Banach space, it is well-known that, for any $k \in \N,$ 
   $2^{\frac{1-k}{n}} \leq e_k(I) \leq 4\cdot 2^{\frac{1-k}{n}},$ where $I:X \rightarrow X$ is the identity map (see \cite{C&S}). 
   A simple modification gives us the following analogous result :
   \begin{thm}  \label{Entorpy of finite dimenal identity map}
   	Let $k \in \N$ and $X$ an $n$-dimensional $\gamma$-Banach space. If $I:X \rightarrow X$ is the identity map, then 
   	$$2^{\frac{1-k}{n}} \leq e_k(I:X \rightarrow X) \leq 4^{1/\gamma}\cdot 2^{\frac{1-k}{n}}.$$
   \end{thm}

\section{Sharp estimates of entropy numbers in $\gamma$-Banach spaces}

We begin this section by considering the first inner entropy numbers and proving that the constant $2^{1/\gamma-1}$ in (\ref{Intro-First-Outter-Inner-Entropy-Estimate}) is sharp.
The following elementary theorem gives us the better result. 
More precisely, we show that there exist a quasi-Banach space $X$, a $\gamma$-Banach space $Y$ and a bounded linear map $T:X \rightarrow Y$ such that $0 \neq f_k(T) = 2^{1/\gamma-1 } \|T\|$ for all $k \in \N.$


	 
	  \begin{thm} \label{SharpConstant2}
	  	Let $X$ be a quasi-Banach space, $Y$ a $\gamma$-Banach space and $T \in \mathcal{B}(X,Y)$. 
	  	Then $ f_k(T) \leq 2^{1/\gamma-1}\|T\|$ for all $k \in \N$, and the constant  $2^{1/\gamma-1}$ is sharp.   
	  \end{thm}
	  
	  \begin{proof}
	  	Let $I:\ell_\gamma \rightarrow \ell_\gamma$ be the identity map.
	  	Since $\|e_i - e_j | \ell_\gamma \| = 2^{1/\gamma} = 2(2^{1/\gamma-1})$ for every distinct standard unit vectors $e_i$ and $e_j$ with $i,j \in \{1,2,...,{2^{k-1}+1}\},$ it follows that
	  	$f_k(I) \geq 2^{1/\gamma-1} = 2^{1/\gamma-1} \|I\|.$
	  	Therefore, the constant  $2^{1/\gamma-1}$ is best possible.  
	  \end{proof}



	Next, the sharpness of the inequality $2^{1-1/\gamma}\|T\| \leq e_1(T)$ is considered.
	Note that for any $a,b \geq 0$ and $\gamma \in (0,1],$ the following elementary inequality holds
	\begin{equation}
		a^\gamma + b^\gamma \geq (a+b)^\gamma.
	\end{equation}
	To show that the constant $2^{1-1/\gamma}$ is best possible, another reasonable $\gamma$-Banach space will be constructed. 
	
 \begin{thm} \label{SharpConstant1}
 	Let $X$ be a quasi-Banach space, $Y$ a $\gamma$-Banach space and $T \in \mathcal{B}(X,Y)$. Then $2^{1-1/\gamma}\|T\| \leq e_1(T) \leq \|T\|$ and the constant  $2^{1-1/\gamma}$ is sharp.
 \end{thm}

\begin{proof}
		The estimate from above is obvious. 
		On the other hand, we suppose that $T(B_X) \subseteq y+\varepsilon B_Y$ for some $y \in Y$ and $\varepsilon > 0.$ 
		 If $x \in B_X$ then $Tx = y+\varepsilon z_1$ and $T(-x) = y+\varepsilon z_2$ for some $z_1,z_2 \in B_Y,$ and 
		 hence $$2^\gamma \|Tx|Y\|^\gamma = \|Tx - T(-x) |Y\|^\gamma = \varepsilon^\gamma \|z_1 - z_2 |Y\|^\gamma  \leq 2\varepsilon^\gamma.$$
		 This implies that $2^{1-1/\gamma}\|T\| \leq e_1(T).$	
		Next, we will prove that  the constant  $2^{1-1/\gamma}$ is best possible. 
		Let 
		\begin{center}
				\begin{tabular}{p{0.15cm}p{0.05cm}lp{0.15cm}p{0.05cm}l}
					$F_1$ &$=$ &$(-\infty,0] \times [0,\infty),$ &$F_2$ &$=$ &$[0,\infty) \times (-\infty, 0]$ \\
					$G_1$ &$=$ &$(0,\infty) \times (0,\infty),$ &$G_2$ &$=$ &$(-\infty, 0) \times  (-\infty,0).$
				\end{tabular}
		\end{center}
		Define $\varphi : \R^2 \rightarrow [0,\infty)$ by
		$$\varphi (x)= 	\begin{cases}
											\displaystyle |x_1|+|x_2|    &\text{if } x \in F_1 \cup F_2,\\
											\displaystyle \left( |x_1|^\gamma + |x_2|^\gamma \right)^{1/\gamma}      &\text{if } x \in G_1 \cup G_2.
									\end{cases}
		$$
		for all $x = (x_1,x_2) \in \R^2.$
		We will show that $\varphi$ is a $\gamma$-norm.
		Let $x=(x_1,x_2), y=(y_1,y_2) \in \R^2$.
		It is clear that $\varphi(x) = 0$ if and only if $x = 0.$ 
		The homogeneous property follows easily from the properties of absolute value. 
		Next we will investigate the $\gamma$-triangle inequality. 
		The remaining task is to verify that $\varphi^\gamma(x+y) \leq \varphi^\gamma(x) +\varphi^\gamma(y).$
		In the proof, we will distinguish the following cases:
		\begin{enumerate} 
			\item $x+y \in F_1 \cup F_2$
			\item $x+y \in G_1 \cup G_2$
			 \begin{enumerate} 
			 	\item $x,y \in G_1 \cup G_2$
			 	\item $x \in G_1 \cup G_2$ and $y \in F_1 \cup F_2$
			 	\item $x \in F_1 \cup F_2$ and $y \in G_1 \cup G_2$
			 	\item $x,y \in F_1 \cup F_2$
			 \end{enumerate}
		\end{enumerate}		
		{\it Case} 1. $x+y \in F_1 \cup F_2$.
		Since $0<\gamma \leq 1,$ it follows that 
		\begin{align*}
			\varphi^\gamma(x+y) &= (|x_1+y_1|+|x_2+y_2|)^\gamma \\
					&\leq |x_1+x_2|^\gamma+|y_1+y_2|^\gamma \\
					&\leq (|x_1|+|x_2|)^\gamma+(|y_1|+|y_2|)^\gamma \\
					&\leq \varphi^\gamma(x) +\varphi^\gamma(y).
		\end{align*}
		{\it Case} 2. $x+y \in G_1\cup G_2.$ 
		First of all, if $x,y \in G_1\cup G_2,$ we are done because the provided norm on $G_1\cup G_2$ is exactly the $\ell_{\gamma}^2$-norm.
		Next, let us assume that $x \in G_1\cup G_2$ and $y \in F_1 \cup F_2.$ 
		Due to the symmetry, we may assume that $x \in G_1$ and $y \in F_1.$
		Then $x_1,x_2 > 0, ~y_1 \leq 0$ and $y_2 \geq 0$. This implies that $x+y \in G_1,$ and hence
		$$
			\varphi^\gamma(x+y) = |x_1+y_1|^\gamma + |x_2+y_2|^\gamma
					\leq |x_1|^\gamma + |x_2|^\gamma+|y_2|^\gamma\
					\leq \varphi^\gamma(x) +\varphi^\gamma(y).
		$$
		The case $x \in F_1 \cup F_2$ and $y \in G_1\cup G_2$ can  be proved in the same way. 
		Finally, we deal with the case $x,y \in F_1 \cup F_2.$ 
		Without loss of generality we may assume that $x+y \in G_1$ and $x \in F_1$. Then $y \in F_2,$ so that $x_1,y_2 \leq 0$ and $x_2,y_1 \geq 0$.
		Consequently,
		$$	\varphi^\gamma(x+y) = |x_1+y_1|^\gamma + |x_2+y_2|^\gamma \leq  |y_1|^\gamma + |x_2|^\gamma \leq  \varphi^\gamma(x) +\varphi^\gamma(y).$$
		Therefore, $\varphi$ is a $\gamma$-norm.
		
		Now let $Y_0 := (\R^2, \omega),$ where $\omega : \R^2 \rightarrow [0,\infty)$ defined by
		\begin{equation} \label{DefiOmega}
			\omega(x_1,x_2)= 	\begin{cases}
													\displaystyle |x_1|    &\text{if }  |x_1| > |x_2|,\\
													\displaystyle \left( \left| \frac{x_1+x_2}{2}\right| ^\gamma + \left| \frac{x_1-x_2}{2}\right| ^\gamma \right)^{1/\gamma}      &\text{if } |x_1| \leq |x_2|
											\end{cases}	
		\end{equation} 
		for $(x_1,x_2) \in \R^2.$
		Since $\omega$ is obtained from $\varphi$ by rotating through $\frac{\pi}{4}$ radians and scaling by $\frac{1}{\sqrt{2}},$ we have that $\omega$ is also a $\gamma$-norm. 
							\begin{figure}[h!] 
							    	\includegraphics[scale=0.25]{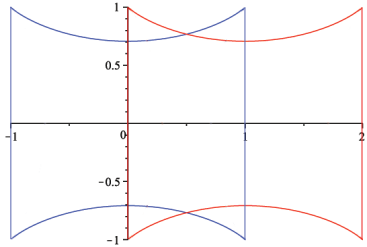}
							    	\caption{$B_{Y_0}$ (left) and $(1,0)+B_{Y_0}$ (right), $\gamma=\frac{2}{3}$.}
							\end{figure}
		
		Define a linear operator $\widetilde{T}: \R \rightarrow Y_0$ by $\widetilde{T}(x)=\left( 0, x\right)$ for all $x \in \R.$ 
		One can see that $\|\widetilde{T}\| = 2^{1/\gamma-1}$. 
			As $\widetilde{T}([-1,1]) = \{(0,x) : -1 \leq x \leq 1\}$ and it can be covered by a unit ball with the center at $(1,0)$ (see Figure 1 for the case $\gamma=\frac{2}{3}$), we have $e_1(\widetilde{T}) \leq 1 = 2^{1-1/\gamma}\|\widetilde{T}\|.$ 
			Therefore, the constant $2^{1-1/\gamma}$ is sharp.
\end{proof}

 The following question might come to readers' mind 
 \lq \lq  Is it possible to obtain the desired sharpness by using only the usual $\ell_{\gamma}$ space?\rq \rq~
 Unfortunately, the next results tell us that the usual $\ell_{\gamma}$ space does not work.

\begin{prop} \label{Prop1-MotivationIneq}
	Let $0<\gamma <1<\beta < \alpha < 2^{1/\gamma -1}$ and  $x=(x_i)_{i=1}^{\infty}, y=(y_i)_{i=1}^{\infty} \in \ell_\gamma$.
	If $\|x|\ell_\gamma\| = 1$ and $\|y|\ell_\gamma\| \leq \beta,$ then  there is a constant $A(\alpha,\beta,\gamma) >0$ such that 
	\begin{equation}  \label{Eq1-MotivationIneq}
		\sum_{i=1}^{\infty} |2^{1/\gamma -1} x_i - y_i|^{\gamma} + |2^{1/\gamma -1} x_i + y_i|^{\gamma} \geq A(\alpha,\beta,\gamma) >2.
	\end{equation}
\end{prop}
\begin{proof}
	Without loss of generality, we assume that $x_i,y_i \geq 0$ for all $i \in \N.$
	Let $I_1 := \{ i \in \N : y_i > \alpha x_i \}$ and $I_2 := \{ i \in \N : y_i \leq \alpha x_i \}.$
	Since $0<\gamma <1,$ we have that
	\begin{equation} \label{Eq2-MotivationIneq}
		\sum_{i \in I_1} |2^{1/\gamma -1} x_i - y_i|^{\gamma} + |2^{1/\gamma -1} x_i + y_i|^{\gamma}
			\geq 	\sum_{i \in I_1} |2^{1/\gamma -1} x_i +2^{1/\gamma -1} x_i|^{\gamma}
			= 2 	\sum_{i \in I_1} |x_i|^{\gamma}.
	\end{equation}
	For any fixed $a >0,$ let us consider a function $g_a:[-a,a] \rightarrow \R$ defined by 
	$$g_a(t) = (a-t)^{\gamma} + (a+t)^{\gamma}$$
	for any $t \in [-a,a].$
	Then $g_a^{\prime}(t) = \gamma \left[ (a+t)^{\gamma -1} - (a-t)^{\gamma - 1} \right] $ and hence $g_a^{\prime}(t) \leq 0$ for any $t \geq 0.$
	Moreover, $\displaystyle \min_{t \in [-a,a]} g_a(t) = g_a(0)$ because $g_a$ is an even function. 
	This implies that $g_a$ is decreasing on $[0,a].$
	Using this remark, we obtain that
	\begin{align}  \label{Eq3-MotivationIneq}
			\sum_{i \in I_2} |2^{1/\gamma -1} x_i - y_i|^{\gamma} + |2^{1/\gamma -1} x_i + y_i|^{\gamma}
				&\geq 	\sum_{i \in I_2} |2^{1/\gamma -1} x_i - \alpha x_i|^{\gamma} + |2^{1/\gamma -1} x_i + \alpha x_i|^{\gamma} \notag \\
				&=  \left[  (2^{1/\gamma -1} - \alpha)^\gamma + (2^{1/\gamma -1} + \alpha)^\gamma \right]   \sum_{i \in I_2} |x_i|^{\gamma}
	\end{align}	
	 Since $\sum_{i \in I_1} |x_i|^\gamma + \sum_{i \in I_2} |x_i|^\gamma =1$ and 
	$\beta^\gamma \geq \sum_{i \in I_1} |y_i|^\gamma \geq \alpha^\gamma\sum_{i \in I_1} |x_i|^\gamma,$ we have
	\begin{equation} \label{Eq4-MotivationIneq}
		\sum_{i \in I_2} |x_i|^\gamma \geq 1-\left( \frac{\beta}{\alpha}\right) ^{\gamma} >0.
	\end{equation}
	We note also that 
	\begin{equation} \label{Eq5-MotivationIneq}
		(2^{1/\gamma -1} - \alpha)^\gamma + (2^{1/\gamma -1} + \alpha)^\gamma >2.
	\end{equation}
	 It follows from (\ref{Eq2-MotivationIneq}), (\ref{Eq3-MotivationIneq}), (\ref{Eq4-MotivationIneq}) and (\ref{Eq5-MotivationIneq})  that
	\begin{equation*} 
		\sum_{i=1}^{\infty} |2^{1/\gamma -1} x_i - y_i|^{\gamma} + |2^{1/\gamma -1} x_i + y_i|^{\gamma} 
			\geq A(\gamma, \beta, \alpha)   >2,
	\end{equation*}	
	where $A(\alpha,\beta,\gamma) :=2\left( \frac{\beta}{\alpha}\right) ^{\gamma} +[(2^{1/\gamma -1} - \alpha)^\gamma + (2^{1/\gamma -1} + \alpha)^\gamma]\left(1- \left( \frac{\beta}{\alpha}\right) ^{\gamma}\right).$
\end{proof}
The next result is a direct consequence of the above proposition.
\begin{cor} \label{Corollary-MotivationIneq}
		Let  $0<\gamma <1<\beta < \alpha < 2^{1/\gamma -1}$ and  $x,y \in \ell_\gamma$.
		If $\|x|\ell_\gamma\| = 1$ and $\{ -2^{1/\gamma -1}x, 0, 2^{1/\gamma -1}x \}$ $\subseteq y + \varepsilon B_{\ell_{\gamma}}$ for some $\varepsilon >0,$  then $\varepsilon \geq  B(\alpha,\beta,\gamma) >1,$
		where $B(\alpha,\beta,\gamma) := \min \left(  \beta, \left( \frac{A(\alpha,\beta,\gamma) }{2} \right)^{1/\gamma}  \right)$ and $A(\alpha,\beta,\gamma)$ is the constant defined as above.  
\end{cor}

\begin{proof}
	Assume that $\|x|\ell_\gamma\| = 1$ and $\{ -2^{1/\gamma -1}x, 0, 2^{1/\gamma -1}x \} \subseteq y + \varepsilon B_{\ell_{\gamma}}$ for some $\varepsilon >0.$
	If $\|y|\ell_\gamma\| > \beta,$ then $\varepsilon > \beta$ because $0 \in y + \varepsilon B_{\ell_{\gamma}}.$
	On the other hand, if $\|y|\ell_\gamma\| \leq \beta,$ Proposition \ref{Prop1-MotivationIneq} implies that 
	$$2\varepsilon^\gamma \geq \sum_{i=1}^{\infty} |2^{1/\gamma -1} x_i - y_i|^{\gamma} + |2^{1/\gamma -1} x_i + y_i|^{\gamma} \geq  A(\alpha,\beta,\gamma) >2,$$
	where $A(\alpha,\beta,\gamma)$ is the constant defined in the proof of Proposition \ref{Prop1-MotivationIneq}.
	The result follows.
\end{proof}

Now we are going to show that it is impossible to obtain the sharpness of the inequality $2^{1-1/\gamma}\|T\| \leq e_1(T)$ by considering only $\ell_{\gamma}$-spaces.

\begin{thm} \label{Prop2-MotivationIneq}
	Let  $0<\gamma <1<\beta < \alpha < 2^{1/\gamma -1}$  and let $X$ be any quasi-Banach space.
	 Then there is a constant $C(\alpha,\beta,\gamma)  > 2^{1-1/\gamma}$ such that for any bounded linear operator  $T:X \rightarrow \ell_\gamma$ the following inequality holds   
	$$e_1(T) \geq C(\alpha,\beta,\gamma)  \|T\|.$$
\end{thm}
\begin{proof}
The result is obvious if $T$ is a zero operator. We suppose that $T \neq 0.$
 Let us fix $\delta>0.$ 
 Then there exists $z \in X$ such that $\|z | X\|=1$ and 
 \begin{equation}
 	\|T(z) | \ell_\gamma\| \geq \|T\| - \delta.
 \end{equation}
 Define $S :\R \rightarrow \ell_\gamma$ by $\displaystyle S(\xi) =  \xi T(z) $ for all $\xi \in \R.$
 Then $S ([-1,1]) \subseteq T(B_X)$ and $\|S\| = \|T(z) |\ell_\gamma\| \leq \|T\|.$
 Let $\rho >0$ be such that $\displaystyle \rho > e_1\left(  \frac{2^{1/\gamma -1}T}{\|T(z) |\ell_\gamma\|} \right) .$
 Then there exists $y \in \ell_\gamma$ such that 
 \begin{equation}
 	\frac{2^{1/\gamma -1}}{\|T(z) |\ell_\gamma\|}T(B_X) \subseteq y + \rho B_{\ell_\gamma}.
 \end{equation}
 We note that 
 $$\displaystyle \left\lbrace - \frac{2^{1/\gamma -1}}{\|T(z) |\ell_\gamma\|}T(z), 0, \frac{2^{1/\gamma -1}}{\|T(z) |\ell_\gamma\|} T(z)\right\rbrace  \subseteq \frac{2^{1/\gamma -1}T}{\|T(z) |\ell_\gamma\|} S([-1,1]) \subseteq \frac{2^{1/\gamma -1}T}{\|T(z) |\ell_\gamma\|}T(B_X).$$ 
 By Corollary \ref{Corollary-MotivationIneq}, we have $\rho \geq B(\alpha,\beta,\gamma) >1,$ where $B(\alpha,\beta,\gamma)$ is the constant defined as above.
 Letting $\rho \rightarrow e_1\left(  \frac{2^{1/\gamma -1}T}{\|T(z) |\ell_\gamma\|} \right) ,$ we obtain that $e_1\left(  \frac{2^{1/\gamma -1}T}{\|T(z) |\ell_\gamma\|} \right)  \geq B(\alpha,\beta,\gamma) .$
 This implies that 
 $$e_1(T) \geq  2^{1-1/\gamma}B(\alpha,\beta,\gamma)  \|T(z) |\ell_\gamma\| \geq 2^{1-1/\gamma}B(\alpha,\beta,\gamma)   (\|T\| - \delta).$$
 Now let $\delta \rightarrow 0,$ we finally get $e_1(T) \geq C(\alpha,\beta,\gamma)  \|T\|,$ where $C(\alpha,\beta,\gamma)  := 2^{1-1/\gamma} B(\alpha,\beta,\gamma).$
\end{proof}


Next, we deal with the question motivated by Theorem \ref{SharpConstant1} that \lq \lq Is it possible to find a Banach space $X$, a $\gamma$-Banach space $Y$ and $T \in \mathcal{B}(X,Y)$ such that $2^{1-1/\gamma}\|T\| = e_k(T)$ for all $k \in \N ?$\rq \rq~
We note that when $Y$ is a Banach space $(\gamma =1)$ a positive answer was provided by Hencl (see \cite{Hencl2003}). 
In Theorem \ref{SharpConstant3}, we also give the positive answer when $Y$ is a $\gamma$-Banach space.
 To prove this result, we need a few auxiliary results. 
First, we introduce a special type of basis which is modified for $\gamma$-Banach spaces 
; it is slightly different from the definition of unconditional bases which appears in the book of Lindenstrauss and Tzafriri (see \cite{Lind&Tzaf}). 

 \begin{defn}
 	Let $E$ be an $n$-dimensional $\gamma$-Banach space.
 	A basis $\{w_1,...,w_n\}$ for $E$ is called {\it $1$-unconditional} if for every $\alpha_1,\alpha_2,...,\alpha_n \in \R$  the following property holds:
 	$$\left\| \sum_{i=1}^{n}  \alpha_i w_i | E \right\|  = \left\| \sum_{i=1}^{n} |\alpha_i| w_i | E\right\|.$$ 
 
 \end{defn}
 
 
 
Let $E$ be an $n$-dimensional $\gamma$-Banach space and $\{w_1,...,w_n\}$ a $1$-unconditional basis for $E$. 
 Let $X_1,X_2,...,X_n$ be any Banach spaces. Define $\tau : \prod_{i=1}^{n} X_i \rightarrow [0,\infty)$ by
 \begin{equation} \label{DefOfPhi}
 	\tau (x) := \left\| \left(  \|x_1|X_1\|, \|x_2|X_2\|,...,\|x_n|X_n\| \right) |E \right\| 
 \end{equation}
 for $x = (x_1,x_2,...,x_n) \in \prod_{i=1}^{n} X_i .$
 We note that if $E$ is an $n$-dimensional Banach space, then the map $\tau$ defined as above is a norm. 
 However, if $E$ is an $n$-dimensional $\gamma$-Banach space, an additional restriction on the space $E$ is required to ensure that $\tau$ is a $\gamma$-norm.
	
	
 	Let
 	 $u=(u_i)_{i=1}^n,$ $v=(v_i)_{i=1}^n \in [0,\infty)^n.$ 
 	For each $i \in \{1,2,...,n\},$ let $m_i := \min(u_i,v_i)$ and $M_i := \max(u_i,v_i).$ 
 	Set \begin{equation}
 			Q_{u,v} := \prod_{i=1}^{n} [M_i-m_i, M_i+m_i].
 		\end{equation}
 		
 \begin{defn}
 	Let $E$ be an $n$-dimensional $\gamma$-Banach space with a $1$-unconditional basis.
 	We say that $E$ satisfies the {\it condition $(Q_\gamma)$} if for any $u,v \in [0,\infty)^n$, the following estimate holds:
 	\begin{equation} \label{ConditionQ}
 		\sup_{x \in Q_{u,v}} \|x|E\|^\gamma \leq \|u|E\|^\gamma +\|v|E\|^\gamma.
 	\end{equation}
 \end{defn}

 \begin{prop} \label{SatisfyCondQImplyPNorm}
 	Let $E$ be an $n$-dimensional $\gamma$-Banach space with a $1$-unconditinal basis.	
 	If $E$ satisfies the condition $(Q_\gamma)$, then the map $\tau$ defined in (\ref{DefOfPhi})  is a $\gamma$-norm.
 \end{prop}
 \begin{proof}
 	It suffices to show that the $\gamma$-triangle inequality holds. 
 	Let $x=(x_i)_{i=1}^n, y = (y_i)_{i=1}^n \in \prod_{i=1}^{n} X_i.$
 	Put $ u=\left(  \|x_i|X_i\| \right)_{i=1}^{n}$ and $v=\left(  \|y_i|X_i\| \right)_{i=1}^{n}.$ 
 	For each $i = 1,2,...,n,$ since $X_i$ is a Banach space,  we have that
 	$$M_i - m_i =\left| \|x_i|X_i\|- \|y_i|X_i\| \right| \leq \|x_i + y_i|X_i\| \leq \|x_i|X_i\| + \|y_i|X_i\| = M_i+m_i.$$
 	Hence,
 	 $z = \left(  \|x_i + y_i|X_i\| \right)_{i=1}^{n} \in Q_{u,v}.$
 	It follows from (\ref{ConditionQ}) that
 	$$\tau^\gamma(x+y) =  \left\|z| E \right\|^\gamma \leq \|u|E\|^\gamma +\|v|E\|^\gamma = \tau^\gamma(x) + \tau^\gamma(y).$$
 \end{proof}
  
 \begin{rem}
 	If there are Banach spaces $X_1,X_2,...,X_n$ with $\dim X_i \geq 2, ~i=1,2,...,n$ so that $\tau$ is a $\gamma$-norm, then $E$ will satisfy the condition $(Q_\gamma)$.
 \end{rem}

 
 Now we are ready to prove our main result.
 \begin{thm} \label{SharpConstant3}
 There exist a Banach space $X$, a $\gamma$-Banach space $Y$ and $T \in \mathcal{B}(X,Y)$ such that $2^{1-1/\gamma}\|T\| = e_k(T)$ for all $k \in \N.$
 \end{thm}
 
 \begin{proof} 
 Let $E_\gamma := (\R^2, \omega),$ where $\omega$ is the $\gamma$-norm defined as (\ref{DefiOmega}).
 First, we will show that $E_\gamma$ satisfies the condition $(Q_\gamma)$.
 Let $u = (u_1,u_2), v = (v_1,v_2) \in [0,\infty)^2.$
 Fix $a,b \in [0,\infty).$ 
 Let $\alpha:[0,\infty) \rightarrow [0,\infty)$ be a function defined by $\alpha(t) = \omega((a,b+t))$ and let $\beta:[0,\infty) \rightarrow [0,\infty)$ be a function defined by $\beta(t) = \omega((a+t,b)).$ 
 Notice that $\alpha$ is an increasing function. In addition, $\beta$ decreases on an interval and then goes up linearly; in other words, if $a \geq b$ then $\beta$ is increasing; on the other hand, if $a < b$ then $\beta$ decreases on $[a,b]$ and increases on $[b,\infty).$ 
 Thus, the following estimate holds: 
 \begin{equation} \label{Eq1-SharpConstant3}
 \sup_{x \in Q_{u,v}} \omega^\gamma(x) \leq \max(\omega^\gamma((M_1+m_1,M_2+m_2)) , \omega^\gamma((M_1-m_1,M_2+m_2)) )
 \end{equation}
 It is clear that $\omega^\gamma((M_1+m_1,M_2+m_2)) = \omega^\gamma(u+v) \leq \omega^\gamma(u) + \omega^\gamma(v).$
 Next, we will estimate $\omega^\gamma((M_1-m_1,M_2+m_2)).$ 
 If $M_1 = u_1,$ we put $\tilde{u} = u$ and $\tilde{v} = (-v_1,v_2)$. 
 Then, 
 \begin{equation*}
 \omega^\gamma((M_1-m_1,M_2+m_2)) 
 	= \omega^\gamma(\tilde{u}+\tilde{v}) 
 	\leq  \omega^\gamma(\tilde{u})+ \omega^\gamma(\tilde{v}) 
 	= \omega^\gamma(u)+ \omega^\gamma(v).
 \end{equation*}
 Similarly, if $M_1 = v_1,$ we have $\omega^\gamma((M_1-m_1,M_2+m_2)) \leq \omega^\gamma(u)+ \omega^\gamma(v).$
 These estimates and (\ref{Eq1-SharpConstant3}) imply that $E_\gamma$ satisfies the condition $(Q_\gamma)$.

Now, let $X:=\ell_p$ for some $1\leq p\leq \infty.$ 
	It is clear that $e_k(I:X \rightarrow X) \leq  1$ for all $k \in \N.$
   	 	Suppose that there are $k_0 \in \N$ and $\alpha > 0$ such that $e_{k_0}(I:X \rightarrow X) < \alpha < 1.$
   	 	Let $n_0 \in \N$ be such that $\alpha < 2^{\frac{1-k_0}{n_0}}.$
   	 		Theorem \ref{Entorpy of finite dimenal identity map} implies that  
   	 	 \begin{equation}
   	 	  2^{\frac{1-k_0}{n_0}} \leq e_k(I: \ell_p^{n_0} \rightarrow \ell_p^{n_0}) \leq \|P\| e_k(I:X \rightarrow X) \|J\| = e_k(I:X\rightarrow X) < \alpha ,
   	 	  \end{equation}
   	 	where $J: \ell_p^{n_0}  \rightarrow X$ and $P:X \rightarrow \ell_p^{n_0}$ are defined by
   		$J(x_i)_{i=1}^{n_0} = (x_1,...,x_{n_0},0,0...)$ and $P(x_i)_{i=1}^{\infty} = (x_1,...,x_{n_0}).$   	 	
   	 	This is a contradiction. Hence, $e_k(I:X \rightarrow X) = 1$ for all $k \in \N.$  
Next, define $\vartheta: \R \times X \rightarrow [0,\infty)$  by
\begin{equation} \label{Eq-GammaNorm}
	 \vartheta(\xi,x) = \omega((|\xi|, \|x|X\|))
\end{equation}
 for $(\xi,x) \in \R \times X.$
 By Proposition \ref{SatisfyCondQImplyPNorm}, $\vartheta$ is a $\gamma$-norm.
Let $Y:= (\R \times X, \vartheta)$ and define $T:X \rightarrow Y$ by $T(x) = (0,x)$ for all $x \in X.$ 
 For each $x \in B_X,$ we have $\vartheta(Tx) = \omega(0, \|x|X\|) = 2^{1/\gamma-1}\|x|X\|,$ which implies that $\|T\| = 2^{1/\gamma-1}.$ 
  Let $P$ be the projection of $Y$ onto $X.$ 
Then,
 	\begin{equation} \label{Eq1-NoteThm3.1.10}
 		1 = e_k(I) = e_k(PT) \leq \|P\|e_k(T) = e_k(T). 
 	\end{equation}	
 	On the other hand, since $T(B_X) \subseteq (1,0) + B_{Y},$ it turns out that
 		$e_k(T) \leq e_1(T) \leq 1.$ 
 	Combining with (\ref{Eq1-NoteThm3.1.10}), we obtain that  
 	$e_k(T) = 1 = 2^{1-1/\gamma} \|T\|.$ 
 \end{proof}

    Next we consider the metric  injection and metric surjection properties of entropy numbers.

      
     \begin{defn}
     	Let $\widetilde{X}$ and $X$ be quasi-Banach spaces.
     	A continuous linear map $\varsigma$ from $\widetilde{X}$ onto $X$ is called a {\it metric surjection} if the image of the unit ball in $\widetilde{X}$ under $\varsigma$ is the unit ball in $X$. 
     \end{defn}
     
     If $\widetilde{X},X$ and $Y$ are Banach spaces, it is known that $e_{k}(T\varsigma) = e_{k}(T)$ and  $f_{k}(T\varsigma) = f_{k}(T)$ (see \cite{C&S}, p.12-13). A careful observation gives us the following similar result. 
       
     \begin{prop} \label{SurjectivityOfe}
     	Let $\widetilde{X},X$ be quasi-Banach spaces and $Y$ a $\gamma$-Banach space.
     	Let $T \in \mathcal{B}(X,Y)$ and $\varsigma:\widetilde{X} \rightarrow X$ be a metric surjection. Then, for each $k \in \N,$ 
     	\begin{center}
     		$e_{k}(T\varsigma) = e_{k}(T)$ and  $f_{k}(T\varsigma) = f_{k}(T).$ 
     	\end{center}
     \end{prop}

   \begin{defn} \label{Def-Metric-Injection}
   	Let $Y$ and $\widetilde{Y}$ be quasi-Banach spaces.
   	A continuous linear map $\iota$ from $Y$ into $\widetilde{Y}$ is called a {\it metric injection} if $\|y | Y\| = \|\iota y | \widetilde{Y}\|$ for all $y \in Y.$  
   \end{defn}
   
  
  
  
   If $X, Y, \widetilde{Y}$ are Banach spaces, then  $e_{k}(\iota T) \leq e_{k}(T) \leq 2e_k(\iota T)$ and  $f_{k}(\iota T) = f_{k}(T)$ (see \cite{C&S}, p.13). Moreover, the constant $2$ cannot be reduced (see also \cite{C&S}, p.125). 
   We now consider the case that our spaces are quasi-Banach spaces. A similar result is obtained. 
   
  \begin{prop} \label{InjectivityOff}
  	Let $X$ be a quasi-Banach space and let $Y,\widetilde{Y}$ be $\gamma$-Banach spaces.
  	Let $T \in \mathcal{B}(X,Y)$ and $\iota:Y \rightarrow \widetilde{Y}$ be a metric injection. Then for each $k \in \N,$ 
  	\begin{equation} \label{Eq-InjectivityOff}
  		e_{k}(\iota T) \leq e_{k}(T) \leq 2^{1/\gamma}e_k(\iota T) ~\text{and}~ f_k(T) = f_k(\iota T).
  \end{equation}
  \end{prop} 
   \begin{proof}
   	It is clear that $ e_k(\iota T) \leq e_k(T) $ and $f_k(T) \leq f_k(\iota T).$ 
   	Let $0<\rho <f_k(\iota T).$ Then there are $x_1,...,x_{2^{k-1}+1} \in B_X$ such that $\| \iota T x_i - \iota T x_j | \widetilde{Y} \| > 2 \rho$ for $i \neq j.$
   	Since $\| \iota T x_i - \iota T x_j | \widetilde{Y} \| = \| \iota (T x_i - T x_j) | \widetilde{Y} \| = \|  T x_i -  T x_j | Y \|,$ we have $f_k(T) \geq \rho$; therefore, $f_k(T) \geq f_k(\iota T).$ 
   	Finally, it follows from (\ref{EqRelationInner&Outer}) that
   		$e_k(T) \leq 2f_k(T)= 2f_k(\iota T) \leq  2^{1/\gamma}e_k(\iota T).$
   \end{proof}
   

   Recall that $c_0 := \{(x_i)_{i=1}^\infty \in \ell_\infty :  \lim_{i \rightarrow \infty} x_i = 0 \}$ is a subspace of $\ell_\infty.$ 
   The following example which is a modification of the results in \cite{C&S}, p.125, shows that the constant $2^{1/\gamma}$ in (\ref{Eq-InjectivityOff}), in fact, cannot be reduced.
    
   \begin{example}  \label{Ex-constant-cannot-be-reduced}
   		Let $T_0 :\ell_1 \rightarrow (\{0\} \times c_0, \vartheta)$ and $T_{\infty} :\ell_1 \rightarrow (\R \times \ell_\infty, \vartheta)$ be the maps defined by $T_0(x) = (0,x)$ and $T_\infty(x) = (0,x)$ for any $x \in \ell_1$ Note that $\vartheta$ is the $\gamma$-norm defined as (\ref{Eq-GammaNorm}).
   		First, we will show that $e_k(T_0) = 2^{1/\gamma - 1}$ for all $k \in \N.$ 
   		Let $k \in \N$ and $\varepsilon > e_k(T_0).$ 
   		   		Then there are $z^{(1)}, ..., z^{(2^{k-1})} \in c_0$ such that 
   		   		$$T_0(B_{\ell_1}) \subseteq \bigcup_{i=1}^{2^{k-1}} \left\lbrace  (0,z^{(i)}) + \varepsilon B_{(\{0\} \times c_0, \vartheta)} \right\rbrace .$$
   		   			For each $i = 1,2,...,2^{k-1},$  there exists $N_i \in \N$ such that $|z^{(i)}_{j}| \leq \varepsilon - e_k(T_0)$ for all $j \geq N_i.$ Let $N = \max\{ N_i : i=1,2,...,2^{k-1} \}$ and let $e_N$ be a standard unit vector in $\ell_{1}$. 
   		   		   	Thus there is $i_0 \in \{ 1,2,...,2^{k-1} \}$ such that 
   		   		   		$ \varepsilon \geq \vartheta \left( T_0(e_N) - (0,z^{(i_0)})\right) = \omega \left( ( 0,\|e_N-z^{(i_0)} | c_0\|) \right).$
   		   		   	We note that 
   		   		   	$$ \omega \left( ( 0,\|e_N-z^{(i_0)} | c_0\|) \right) = 2^{1/\gamma - 1} \|e_N-z^{(i_0)} | c_0\| \geq 2^{1/\gamma - 1} |1-z^{(i_0)}_N| \geq 2^{1/\gamma - 1} (1- \varepsilon +e_k(T_0)).$$
   		   		   	Thus,  $(2^{1/\gamma - 1} +1) \varepsilon \geq 2^{1/\gamma - 1} (1+ e_k(T_0)).$
   		   		   	Letting $\varepsilon \rightarrow e_k(T_0)$, we obtain that $e_k(T_0) \geq 2^{1/\gamma - 1}.$
   		 On the other hand, $\vartheta(T_0(x)) = \omega(0,\|x | c_0\|) = 2^{1/\gamma - 1} \|x | c_0\| \leq 2^{1/\gamma - 1} \|x | \ell_1\|.$
   		 This implies that $e_k(T_0) \leq \|T_0\| = 2^{1/\gamma - 1}.$ Consequently, $e_k(T_0) = 2^{1/\gamma - 1}.$
   		   		   
   		Next, we will show that $\frac{1}{2}	\leq e_k(T_\infty) \leq \frac{1}{2} + \frac{1}{2^{k-1}}$ for all $k \in \{2,3,...\}.$	
   		As $\vartheta \left( T_\infty(e_i) - T_{\infty}(e_j) \right) =\omega \left( (0, \|e_i - e_j | \ell_\infty\|) \right) = 2^{1/\gamma-1}$ for all standard unit vectors $e_i$ and $e_j$ with $i \neq j,$ by (\ref{EqRelationInner&Outer}) we have 
   		\begin{equation}
   			e_k(T_\infty) \geq 2^{1-1/\gamma}f_k(T_\infty) \geq 2^{1-1/\gamma} 2^{1/\gamma - 2} = \frac{1}{2}.
   		\end{equation}
   		Let $I_\infty :\ell_1 \rightarrow \ell_\infty$  be the natural embedding, and
   		define $S_\infty:\ell_\infty \rightarrow (\R \times \ell_\infty, \vartheta)$ by $S_\infty(x) = (0,x)$ for any $x \in \ell_\infty.$
   		Observe that $S_\infty$ is obtained by substituting $X = \ell_\infty$ of the function $T$ in the proof of Theorem \ref{SharpConstant3}, so we have  $e_m(S_\infty) = 1$ for all $m \in \N.$ 
		Since $e_m(I_\infty) \leq \frac{1}{2} + \frac{1}{2^{m-1}}$ for all $m \in \{2,3,...\}$ (see \cite{C&S}, p.125), 
		it follows  that
   		\begin{equation}
   			e_k(T_\infty) \leq e_{1}(S_\infty) e_{k}(I_\infty)
   			   			\leq \frac{1}{2} + \frac{1}{2^{k-1}}
   		\end{equation} 
   		for all $k \in \{2,3,...\}.$		
		
			Now let $\iota: (\{0\}\times c_0, \vartheta) \rightarrow (\R \times \ell_\infty, \vartheta)$ be the natural embedding.
   		   	Then, $\iota T_0 = T_\infty.$
   		   	If $e_k(T_0) \leq \alpha e_k(\iota T_0)$ for some $\alpha >0$, then
   		   	$2^{1/\gamma - 1} \leq \alpha \cdot 2^{-1},$ so $\alpha$ cannot be less than $2^{1/\gamma}.$ 
   \end{example}


\section{Entropy numbers of embeddings between finite dimensional symmetric $\gamma$-Banach spaces}

In this section we give estimates for entropy numbers of embeddings between finite dimensional $\gamma$-Banach spaces with symmetric bases.
Recall that a basis $(x_i)_{i=1}^\infty$ of a $\gamma$-Banach space $X$ is called {\it symmetric} if for any permutation $\pi,$ any $\varepsilon_i \in \{\pm 1\}$ and any $a_i \in \R,$
\begin{equation*}
	\left\|  \sum_{i=1}^{\infty} a_i x_i | X\right\| =  \left\|  \sum_{i=1}^{\infty} \varepsilon_i a_{\pi (i)} x_i | X\right\|.
\end{equation*}
Let $n \in \N$ and let $X$ and $Y$ be $n$-dimensional $\gamma$-Banach spaces with normalised symmetric bases $\{x_i\}_{i=1}^n$ and $\{y_i\}_{i=1}^n$ respectively. 
We consider the natural embedding $id:X \rightarrow Y$ given by
$$id\left( \sum_{i=1}^{n} a_i x_i \right) =  \sum_{i=1}^{n} a_i y_i.$$

The following useful result which was proved by Edmunds and Netrusov can be found in \cite{E&N1998} (Section 4, Lemma 3).
\begin{lem} \label{Lem-Entropy-on-finite-dim} 
	Let $X$ be an $n$-dimensional $\gamma$-Banach space with a symmetric basis $(x_i)_{i=1}^n.$
	Then there is a constant $c(\gamma)$, depending only on $\gamma$, such that 
	\begin{equation}
		e_n(id:X \rightarrow \ell_\infty) \leq e_n(id:X \rightarrow \ell_\infty^n) \leq  \frac{c(\gamma)}{\varphi_X(n)},
	\end{equation}
	where $\varphi_{X}(n) = \left\| \sum_{i=1}^{n} x_i  | X\right\|$ is the {\it fundamental function} of $X.$
\end{lem}

 Next, we are going to estimate entropy numbers of embeddings between finite dimensional symmetric $\gamma$-Banach spaces.
 If $X$ and $Y$ are Banach spaces, this result was proved by Sch\"{u}tt in 1984 (see \cite{Sch}).
\begin{thm} \label{Thm-EntropyInFinteDim-FundamentalFn1}
	Let $X$ and $Y$ be any $n$-dimensional $\gamma$-Banach spaces  with normalised symmetric bases $\{x_i\}_{i=1}^n$ and $\{y_i\}_{i=1}^n$ respectively. 
	Then there are positive constants $c_1,c_2 $ which depend only on $\gamma$ such that, for any $k \geq n,$
	\begin{equation} \label{Eq1-Entropy-on-finite-dim}
		c_1 2^{-k/n} \frac{\varphi_Y(n)}{\varphi_X(n)} \leq e_k(id:X \rightarrow Y) \leq c_2 2^{-k/n} \frac{\varphi_Y(n)}{\varphi_X(n)}.
	\end{equation}
\end{thm}

\begin{proof}
	First, let us prove the upper estimate.
	By Theorem  \ref{Entorpy of finite dimenal identity map}, we have that
	\begin{align} \label{Eq2-Entropy-on-finite-dim}
		e_k(id:X \rightarrow Y) &\leq e_{k-n+1}(id:X \rightarrow X) e_n(id:X \rightarrow Y) \notag\\
			&\leq c_1 2^{\frac{1-(k-n+1)}{n}} e_n(id:X \rightarrow \ell_{\infty}^n) \|id:\ell_{\infty}^n \rightarrow Y\| \notag\\
			&\leq c_2 2^{-k/n} \varphi_{Y}(n) e_n(id:X \rightarrow \ell_{\infty}^n).
	\end{align}
	The desired estimate follows from (\ref{Eq2-Entropy-on-finite-dim}) and Lemma \ref{Lem-Entropy-on-finite-dim}.
 	
 		To prove the lower estimate, we again apply Theorem  \ref{Entorpy of finite dimenal identity map} and the upper estimate so that we have
	\begin{align*}
		2^{\frac{1-(k+n-1)}{n}} &\leq e_{k+n-1}(id:X \rightarrow X) \\
			& \leq e_k(id:X \rightarrow Y) e_n(id:Y\rightarrow X)\\
			& \leq c_3 2^{-n/n} \frac{\varphi_X(n)}{\varphi_Y(n)} e_k(id:X \rightarrow Y). 
	\end{align*}
	Consequently, $e_k(id:X \rightarrow Y) \geq c_4 2^{-k/n} \displaystyle \frac{\varphi_Y(n)}{\varphi_X(n)}$.
\end{proof}


 Finally, a direct consequence of Theorem \ref{Thm-EntropyInFinteDim-FundamentalFn1} gives us two-sided estimates for entropy numbers of embeddings between finite-dimensional Lorentz sequence spaces. 
 Let $p \in (0,\infty), ~r \in (0,\infty].$
 Recall that a {\it Lorentz sequence space} $\ell_{p,r}$ is the set of all bounded sequences $x=(x_i)_{i=1}^\infty$ such that the quasi-norm
  \begin{equation*} 
  			\|x | \ell_{p,r}\|= 	
  				\begin{cases}
  					\displaystyle \left( \sum_{j \in \N}[j^{1/p-1/r} x_j^\ast]^r \right)^{1/r}   &\text{if }  0<r<\infty,\\
  					\displaystyle \sup_{j \in \N} j^{1/p} x_j^\ast      &\text{if } r = \infty.
  				\end{cases}	
  		\end{equation*} 
 is finite.
 Here the sequence $(x_i^\ast)_{i=1}^\infty$ is the non-increasing rearrangement of $(|x_i|)_{i=1}^\infty.$
 The space $\ell_{p,p}$ is simply the sequence space $\ell_p.$
 It is well-known that $c_1 n^{1/p} \leq \varphi_{\ell_{p,r}^n}(n) \leq c_2 n^{1/p}$ for some constants $c_1$ and $c_2$, depending only on $p$ and $r.$ 
The following result is immediate from Theorem \ref{Thm-EntropyInFinteDim-FundamentalFn1}.

\begin{cor} 
	Let $k,n \in \N, ~0<p<q < \infty$ and $0<r,s \leq \infty.$ 
	If $id:\ell_{p,r}^n \rightarrow \ell_{q,s}^n$ is the natural embedding,
	then, for $k \geq n,$ there are positive constants $c_1,c_2$, independent of $k$ and $n$, such that 
	\begin{equation}
		c_1 2^{-k/n}n^{1/q-1/p} \leq e_k(id:\ell_{p,r}^n \rightarrow \ell_{q,s}^n) \leq c_2 2^{-k/n}n^{1/q-1/p}.
	\end{equation}
\end{cor}

   \begin{rem}
  		For $ 0<p < q \leq \infty,$ it is known that there are positive constants $c_1,c_2,$ independent of $k$ and $n,$ such that
  		  \begin{equation} \label{Eq-Schutt-Thm}
  		   		  	c_1 \psi(k,n) \leq	e_k(id:\ell_p^n \rightarrow \ell_q^n) \leq c_2 \psi(k,n),
  		   		\end{equation}
  		  where
  		\begin{equation*} 
  			\psi(k,n)=
  		  		  		\begin{cases}
  		  		  			1 	&~\text{if}~~ 1 \leq k \leq \log_{2} n,\\
  		  		  			\left( \frac{\log_2(1+n/k)}{k} \right)^{1/p-1/q}   &~\text{if}~~ \log_{2} n \leq k \leq n.
  		  		  		\end{cases}
  		\end{equation*}
  		The Banach space case is due to Sch\"{u}tt (see \cite{Sch}, Theorem 1). 
  		The quasi-Banach case can be found in the monograph by Edmunds and Triebel (see \cite{E&T}, Proposition 3.2.2), except the lower estimate in the most interesting range $\log_2 n \leq k \leq n.$
  		This case was solved independently, with different proofs.
		First, we refer to a result which appeared implicitly in the paper of Edmunds and Netrusov in 1998 (see \cite{E&N1998}, Theorem 2); in fact, their result was proved in a more general setting; i.e., they provided two-sided estimates for $e_k(id:X \rightarrow Y),$ where $X$ and $Y$ are $n$-dimensional symmetric quasi-Banach spaces and $k < n/2.$  	
		In addition, the lower estimate of (\ref{Eq-Schutt-Thm}) in the case $\log_2 n \leq k \leq n.$  was also established by Gu\'{e}don and Litvak (see \cite{Gued&Litv}, Theorem 6) in 2000 and by K\"{u}hn (see \cite{Ku}) in 2001.
		We note that more detailed estimates of constants can also be found  in \cite{Gued&Litv}. 
		Moreover, using the interpolation arguments given in \cite{Ku}, the estimate (\ref{Eq-Schutt-Thm}) can be transferred to Lorentz space embeddings $id:\ell_{p,r}^n \rightarrow \ell_{q,s}^n$ whenever $0 < r,s \leq \infty.$
    \end{rem}







\renewcommand{\abstractname}{Acknowledgements}

\begin{abstract}
   I would like to acknowledge the School of Mathematics, University of Bristol for support of this work.
  I am very grateful to Dr.Yuri Netrusov, my supervisor, for his guidance and helpful suggestions.
  Furthermore, I would like to thank Professor David Edmunds for careful reading of the manuscript and for pointing me to the reference \cite{Hencl2003}.
  Finally, I would like to thank the referees for valuable suggestions.
  The 
  author was supported by the Ministry of Science and Technology of Thailand.
\end{abstract}

%

\begin{thebibliography}{[1]}
  
  \bibitem{Aoki}%
  T. Aoki, 
  Locally bounded linear topological spaces,
  Proc. Imp. Acad. Tokyo \textbf{18}, ~No.10 (1942).
  
  \bibitem{C&S}%
  B. Carl and I. Stephani,
  Entropy, compactness and the approximation of operators (Cambridge University Press, New York, 1990).
  
  \bibitem{E&N1998}%
  D.E. Edmunds and Yu. Netrusov, 
  Entropy numbers of embedding of {S}obolev spaces in {Z}ygmund spaces,
  Studia Math. \textbf{128}, 71-102 (1998).

  \bibitem{E&T}%
  D.E. Edmunds and H. Triebel,
  Function Spaces, entropy numbers, differential operators (Cambridge University Press, Cambridge, 1996).
  
  \bibitem{Gued&Litv}%
  O. Guedon and A.E. Litvak,
  Euclidean projections of a $p$-convex body; in: Geometric Aspects of Functional Analysis, Lecture Notes in Mathematics vol. 1745 (Springer, Berlin, 2000), p.\,95-108.
  
  
  \bibitem{Hencl2003}%
  S. Hencl, 
  Measures of non-compactness of classical embeddings of {S}obolev spaces,
  Math. Nachr.  \textbf{258}, 28-43 (2003).
  
  \bibitem{Kol&Ti}%
  A.N. Kolmogorov and V.M. Tichomirov, 
  $\varepsilon$-entropy and $\varepsilon$-capacity of sets in function space ({R}ussian),
  Uspeki Mat. Nauk  \textbf{14(2)}, 3-86 (1959; English transl.: Amer. Math. Soc. Transl. Ser. 2 17:277-364, 1961).
  
  \bibitem{Ku}%
  T. K\"{u}hn, 
  A lower estimate for entropy numbers,
  J. Approx. Theory  \textbf{110}, 120-124 (2001).
  
    \bibitem{Lind&Tzaf}%
    J. Lindenstrauss and L. Tzafriri,
    Classical {B}anach spaces I ({S}equence spaces) (Springer-Verlag, Berlin, 1977).
    
  \bibitem{P0}%
  A. Pietsch,
  Operator ideals (North-Holland, Amsterdam-New York-Oxford, 1980).
  
  \bibitem{P1}%
  A. Pietsch,
  Eigenvalues and s-numbers (Cambridge University Press, Leipzig, 1987).
  
  \bibitem{P2}%
  A. Pietsch,
  History of Banach spaces and linear operators (Birkh\"{a}user, Boston-Basel-Berin, 2007).
  
    \bibitem{Pis}%
    G. Pisier,
    The volume of convex bodies and {B}anach space geometry(Cambridge tracts in mathematics; 94) (Cambridge University Press, Cambridge, 1989).

  \bibitem{Rolewicz}%
  S. Rolewicz, 
  On a certain class of linear metric spaces,
  Bull. Acad. Polon. Sci.  \textbf{5}, 471-473 (1957).
  
  \bibitem{Sch}%
  C. Sch\"{u}tt, 
  Entropy numbers of diagonal operators between symmetric {B}anach spaces,
  J. Approx. Theory  \textbf{40}, 121-128 (1984).
  
  \bibitem{Vitushkin&Henkin}%
  A.G. Vitushkin and G.M. Henkin, 
  Linear superpositions of functions,
  Russian mathematical surveys  \textbf{22}, 77-125 (1967).
  
  
\end{thebibliography}
%

\end{document}